\documentclass[a4paper,12pt]{article}
\usepackage{graphics}
\usepackage{latexsym}
\usepackage{cite}
\usepackage[sf, bf, small]{titlesec}
\usepackage{graphicx,psfrag}
\usepackage{amsmath,amssymb,amsthm,mathtools}

\usepackage{tikz}
\usetikzlibrary{shapes.callouts,decorations.pathmorphing} 
\usetikzlibrary{shadows} 
\usepackage{calc} 
\usetikzlibrary{decorations.markings} 
\tikzstyle{vertex}=[circle, draw, inner sep=0pt, minimum size=6pt] 
\newcommand{\vertex}{\node[vertex]}
\usetikzlibrary{arrows,matrix} 

\newtheorem{Thm}{Theorem}
\newtheorem{Lem}[Thm]{Lemma}
\newtheorem{Prop}[Thm]{Proposition}
\newtheorem{Rem}[Thm]{Remark}
\newtheorem{Cor}[Thm]{Corollary}

\begin{document}
\title{On $(1,2)$-step competition graphs of bipartite tournaments}

\author{\small
\textsc{Jihoon CHOI}
\footnote{Department of Mathematics Education,
Seoul National University, Seoul 151-742, Korea.
\textit{E-mail}: \texttt{gaouls@snu.ac.kr}}
\and
\small \textsc{Soogang EOH}
\footnote{Department of Mathematics Education,
Seoul National University, Seoul 151-742, Korea.
\textit{E-mail}: \texttt{mathfish@snu.ac.kr}}
\and
\small\textsc{Suh-Ryung KIM}
\footnote{Department of Mathematics Education,
Seoul National University, Seoul 151-742, Korea.
\textit{E-mail}: \texttt{srkim@snu.ac.kr}}
\and
\small\textsc{So Jung LEE}
\footnote{Department of Mathematics Education,
Seoul National University, Seoul 151-742, Korea.
\textit{E-mail}: \texttt{sj023@snu.ac.kr}}}
\date{}
\maketitle
\begin{abstract}
In this paper, we study $(1,2)$-step competition graphs of bipartite tournaments. A bipartite tournament means an orientation of a complete bipartite graph. We show that the $(1,2)$-step competition graph of a bipartite tournament has at most one non-trivial component or consists of exactly two complete components of size at least three and, especially in the former, the diameter of the nontrivial component is at most three if it exists.
Based on this result, we show that, among  the connected non-complete graphs which are triangle-free or the cycles of which are edge-disjoint, $K_{1,4}$ is the only graph that can be represented as the $(1,2)$-step competition graph of a bipartite tournament.
We also completely characterize a complete graph and the disjoint union of two complete graphs, respectively, which can be represented as the $(1,2)$-step competition graph of a bipartite tournament.
Finally we present the maximum number of edges and the minimum number of edges which the $(1,2)$-step competition graph of a bipartite tournament might have.
\end{abstract}
\noindent
{\bf Keywords:} bipartite tournament; orientation of complete bipartite graph; $(1,2)$-step competition graphs; $(1,2)$-step competition-realizable.

\section{Introduction}
In this paper, a graph means a simple graph.
For vertices $x$ and $y$ in a digraph $D$, $d_D(x,y)$ denotes the number of arcs in a shortest directed path from $x$ to $y$ in $D$ if it exists.
For positive integers $i$ and $j$, the {\em $(i,j)$-step competition graph} of a digraph $D$, denoted by $C_{i,j}(D)$, is a graph on $V(D)$ where $uv \in E(C_{i,j}(D))$ if and only if there exists a vertex $w \neq u, v$ such that either $d_{D-v}(u,w) \le i$ and $d_{D-u}(v,w) \le j$ or $d_{D-u}(v,w) \le i$ and $d_{D-v}(u,w) \le j$.
The $(1,1)$-step competition graph of a digraph $D$ is the competition graph of $D$. Given a digraph $D$,
the \emph{competition graph} of $D$, denoted by $C(D)$,
is the graph having vertex set $V(D)$ and edge set $\{ uv \mid (u,w) \in A(D), (v,w) \in A(D) \text{ for some } w \in V(D)\}$.
Cohen~\cite{cohen} introduced the notion of competition graph while studying predator-prey concepts in ecological food webs. Cohen's empirical observation that real-world competition graphs are usually interval graphs had led to a great deal of research on the structure of competition graphs and on the relation between the structure of digraphs and their corresponding competition graphs. In the same vein, various variants of competition graph have been introduced and studied, one of which is the notion of $(i,j)$-step competition introduced by Factor and Merz~\cite{tournament}  (see \cite{Cable,m-step,p-competition,RobertsSheng, Scott} for other variants of competition graph).   For recent work on this topic, see~\cite{KA, order, kim1, Kuhl, LiChang, McKay}.

Factor and Merz~\cite{tournament} studied the $(1,2)$-step competition graphs of tournaments. Zhang and Li~\cite{ZR} and Zhang~{\it et al.}~\cite{round} studied the $(1,2)$-step competition graphs of round digraphs. On the other hand, Kim~{\it et al.}~\cite{Kim} studied the competition graphs of orientations of complete bipartite graphs.  In this paper, we study the $(1,2)$-step competition graphs of orientations of complete bipartite graphs, which is a natural extension of their results.

An orientation of a complete bipartite graph is sometimes called a {\it bipartite tournament} and we use whichever of the two terms is more suitable for a given situation throughout this paper.

In Section~2, we derive useful properties of $(1,2)$-step competition graphs of bipartite tournaments.
In Section~3, based on the results obtained in Section~2, we show that, among the  connected non-complete graphs which are triangle-free or the cycles of which are edge-disjoint, $K_{1,4}$ is the only graph that can be represented as the $(1,2)$-step competition graph of a bipartite tournament.
We also completely characterize a complete graph and the disjoint union of complete graphs, respectively, which can be represented as the $(1,2)$-step competition graph of a bipartite tournament.
In Section~4, we present the maximum number of edges and the minimum number of edges which the $(1,2)$-step competition graph of an orientation of $K_{m,n}$ might have.

\section{Properties of $(1,2)$-step competition graphs of bipartite tournaments}

For a digraph $D$, we say that vertices $u$ and $v$ in $D$ {\em $(1,2)$-compete} provided there exists a vertex $w$ distinct from $u, v$ and satisfying one of the following:
\begin{itemize}
\item there exist an arc $(u,w)$ and a directed $(v,w)$-walk of length $2$ not traversing $u$;
\item there exist a directed $(u,w)$-walk of length $2$ not traversing $v$ and an arc $(v,w)$.
\end{itemize}
We call $w$ in the above definition a \emph{$(1,2)$-step common out-neighbor} of $u$ and $v$.
It is said that two vertices \emph{compete} if they have a common out-neighbor.
Thus, $uv \in E(C_{1,2}(D))$ provided $u$ and $v$ compete or $(1,2)$-compete.

An edge in the $(1,2)$-step competition graph of a digraph $D$ is said to \emph{be induced by competition} (resp.\ \emph{$(1,2)$-step competition}) if there exists a common out-neighbor (resp.\ $(1,2)$-step common out-neighbor) of the ends of the edge in $D$.

From the definition of the $(1,2)$-step competition graph of a digraph, we make the following simple but useful observation.

\begin{Prop}\label{prop:deg}
In the $(1,2)$-step competition graph of a digraph $D$, a non-isolated vertex has an out-neighbor in $D$.
\end{Prop}

We present results which will play a key role throughout this paper.

\begin{Prop} \label{prop:structure}
Let $u$ and $v$ be adjacent vertices in the $(1,2)$-step competition graph of an orientation $D$ of $K_{m,n}$.
Then $u$ and $v$ belong to the same partite set if and only if they compete in $D$.
\end{Prop}

\begin{proof}
If $u$ and $v$ compete, then they are adjacent by definition of $(1,2)$-step competition graph. As $D$ is an orientation of $K_{m,n}$, $u$ and $v$ belong to the same partite set of $K_{m,n}$.
To show the converse, suppose that two vertices $u$ and $v$ are adjacent in $C_{1,2}(D)$ and belong to the same partite set. Since $u$ and $v$ are adjacent in $C_{1,2}(D)$, there exists a vertex $w$ in $D$ such that $w$ is a common out-neighbor or a $(1,2)$-step common out-neighbor of $u$ and $v$ by definition, that is, $w$ satisfies one of the following:
\begin{itemize}
\item[(i)] there are arcs $(u,w)$ and $(v,w)$;
\item[(ii)] there exist an arc $(u,w)$ and a directed $(v,w)$-walk of length $2$ not traversing $u$;
\item[(iii)] there exist a directed $(u,w)$-walk of length $2$ not traversing $v$ and an arc $(v,w)$.
\end{itemize}
Only in Case (i), $u$ and $v$ belong to the same partite set of $K_{m,n}$. Therefore $u$ and $v$ compete.
\end{proof}
\noindent The following corollary is the contrapositive of Proposition~\ref{prop:structure}.
\begin{Cor}\label{cor:12compete}
Let $u$ and $v$ be adjacent vertices in the $(1,2)$-step competition graph of an orientation $D$ of $K_{m,n}$.
Then $u$ and $v$ belong to distinct partite sets if and only if they $(1,2)$-compete in $D$.
\end{Cor}

By Corollary~\ref{cor:12compete}, two vertices belonging to distinct partite sets of $K_{m,n}$ can only $(1,2)$-compete in any of its orientations.
The following theorem characterizes a pair of vertices which belong to distinct partite sets of $K_{m,n}$ and $(1,2)$-compete.
\begin{Thm}\label{thm:notcompete}
Let $u$ and $v$ be vertices belonging to distinct partite sets of a bipartite tournament $D$.
Then $u$ and $v$ $(1,2)$-compete in $D$ if and only if $u$ (resp.\ $v$) has an out-neighbor different from $v$ (resp.\ $u$). 
\end{Thm}
\begin{proof}
Let $(V_1,V_2)$ be a bipartition of $D$.
Without loss of generality, we may assume that $u \in V_1$ and $v \in V_2$.
The `only if' part is obviously true. To show the `if' part, suppose that $x$ is an out-neighbor of $u$ different from $v$ and $y$ is an out-neighbor of $v$ different from $u$.
Since $D$ is a bipartite tournament, either $(x,y)$ or $(y,x)$ is an arc in $D$.
If $(x,y)$ (resp.\ $(y,x)$) is an arc, then there exists a directed walk $u \rightarrow x \rightarrow y$ (resp.\ $v \rightarrow y \rightarrow x$) and therefore $u$ and $v$ $(1,2)$-compete.
\end{proof}
\noindent The following corollary is to be quoted frequently in the rest of this paper.

\begin{Cor}\label{cor:edge}
Let $D$ be a bipartite tournament.
Suppose that $(u,v) \in A(D)$ and $d^+(v) \ge 1$.
Then $u$ and $v$ are adjacent in the $(1,2)$-step competition graph of $D$ if and only if $u$ has at least two out-neighbors in $D$.
\end{Cor}
\begin{proof}
Since $(u,v)$ is an arc in $D$, $u$ and $v$ belong to distinct partite sets of $D$.
For the same reason, $u$ has an out-neighbor different from $v$ if and only if $u$ has at least two out-neighbors in $D$.
Moreover, by the hypothesis, $v$ has an out-neighbor.
Since $(u,v)$ is an arc in $D$, it is different from $u$.
Therefore $u$ (resp.\ $v$) has an out-neighbor different from $v$ (resp.\ $u$) if and only if $u$ has at least two out-neighbors in $D$.
Thus the corollary follows from Theorem~\ref{thm:notcompete}.
\end{proof}

\begin{Cor}\label{cor:leasttwo}
Let $G$ be the $(1,2)$-step competition graph of an orientation $D$ of $K_{m,n}$ with a bipartition $(V_1,V_2)$.
Then each vertex has outdegree at least two in $D$ if and only if the edges of $G$ not belonging to $G[V_1] \cup G[V_2]$ induce $K_{m,n}$.
\end{Cor}
\begin{proof}
To show the `only if' part, take two vertices $u$ and $v$ in distinct partite sets of $K_{m,n}$. Then the outdegree of each of $u$ and $v$ is at least two, so $u$ (resp.\ $v$) has an out-neighbor different from $v$ (resp.\ $u$).
Thus, by Theorem~\ref{thm:notcompete}, $u$ and $v$ $(1,2)$-compete and therefore they are adjacent in $G$.
Since $u$ and $v$ are arbitrarily chosen, the edges of $G$ not belonging to $G[V_1] \cup G[V_2]$ induce $K_{m,n}$.
To show the `if' part by contradiction, suppose that there exists a vertex $v$ of outdegree at most one.
If $v$ has no out-neighbor, then $v$ is isolated in $G$ and we reach a contradiction.
Thus $v$ has exactly one out-neighbor, say $u$, and so, by Theorem~\ref{thm:notcompete}, $u$ and $v$ do not $(1,2)$-compete in $D$.
Hence $u$ and $v$ are not adjacent in $G$ and we reach a contradiction. This completes the proof.
\end{proof}

\begin{Lem}\label{lem:neighbor}
Let $D$ be a bipartite tournament with a bipartition $(V_1, V_2)$ and let $uv$ be an edge of $C_{1,2}(D)$ where $u \in V_1$ and $v \in V_2$.
Then $u$ has a neighbor in $V_1$ or $v$ has a neighbor in $V_2$ in $C_{1,2}(D)$.
\end{Lem}
\begin{proof}
Since $u$ and $v$ belong to distinct partite sets, $u$ and $v$ $(1,2)$-compete by Corollary~\ref{cor:12compete}.
Therefore there exist vertices $w$ in $V_1$ and $z$ in $V_2$ such that exactly one of the following is true:
 \begin{itemize}
 \item[(i)] $(u,z), (v,w), (w,z)$ are arcs in $D$;
 \item[(ii)] $(v,w), (u,z), (z,w)$ are arcs in $D$.
 \end{itemize}
If (i) is true, then $w$ becomes a neighbor of $u$ in $C_{1,2}(D)$ and if (ii) is true, then $z$ becomes a neighbor of $v$ in $C_{1,2}(D)$.
Hence the lemma holds.
\end{proof}

Let $G$ be the $(1,2)$-step competition graph of a bipartite tournament with a bipartition $(V_1, V_2)$.
For two edges $x_1x_2$ and $y_1y_2$ of $G$, if there exists a walk between one of $x_1,x_2$ and one of $y_1,y_2$, then we will say that {\em $x_1x_2$ and $y_1y_2$ are connected}.
It is easy to see that if edges $x_1x_2$ and $y_1y_2$ are connected, then any two vertices among $x_1$, $x_2$, $y_1$, $y_2$ are connected.
We say that the edge $x_1x_2$ \emph{links} $V_1$ and $V_2$ if either $x_1 \in V_1$ and $x_2 \in V_2$ or $x_1 \in V_2$ and $x_2 \in V_1$.
Therefore, if the edge $x_1x_2$ does not link $V_1$ and $V_2$, then both $x_1$ and $x_2$ belong to exactly one of $V_1$ and $V_2$.
In the case where the edge $x_1x_2$ does not link $V_1$ and $V_2$, we say that the edge $x_1x_2$ {\em belongs to $V_1$} (resp.\ $V_2$) in $G$ if both of $x_1$ and $x_2$ belong to $V_1$ (resp.\ $V_2$).

\begin{Thm}\label{thm:component}
Let $D$ be a bipartite tournament.
Then $C_{1,2}(D)$ has at most one non-trivial component or consists of exactly two complete components of size at least three.
\end{Thm}
\begin{proof}
Let $(V_1, V_2)$ be a bipartition of $D$.
Suppose that $C_{1,2}(D)$ has at least two non-trivial components.
Take two edges $x_1x_2$ and $y_1y_2$ from distinct non-trivial components.
We first consider the case where each of $x_1x_2$ and $y_1y_2$ does not link $V_1$ and $V_2$.
Suppose that $x_1x_2$ belongs to $V_i$ and $y_1y_2$ belongs to $V_j$ for distinct $i$ and $j$.
Without loss of generality, we may assume that $i=1$ and $j=2$.
Since $x_1x_2$ and $y_1y_2$ belong to distinct components, $x_1$ and $y_1$ are not adjacent, which implies that they do not $(1,2)$-compete by definition.
By Proposition~\ref{prop:deg} and Theorem~\ref{thm:notcompete}, $N^+_D(x_1) = \{y_1\}$ or $N^+_D(y_1) = \{x_1\}$.
Without loss of generality, we may assume $N^+_D(x_1) = \{y_1\}$.
Since $y_1y_2$ does not link $V_1$ and $V_2$, by Proposition~\ref{prop:structure}, $y_1$ and $y_2$ have a common out-neighbor, say $x$, in $V_1$.
Since $x_1 \neq x$ by our assumption that $N^+_D(x_1) = \{y_1\}$,
$x$ is a $(1,2)$-step common out-neighbor of $x_1$ and $y_2$ and so $x_1y_2$ is an edge in $C_{1,2}(D)$, which contradicts the assumption that $x_1x_2$ and $y_1y_2$ belong to distinct components.

Now we consider the case where $x_1x_2$ and $y_1y_2$ belong $V_i$ for some $i \in \{1,2\}$.
Without loss of generality, we may assume $i=1$.
Then, by Proposition~\ref{prop:structure}, $x_1$ and $x_2$  have a common out-neighbor $z_1$ in $V_2$, and $y_1$ and $y_2$ have a common out-neighbor $z_2$ in $V_2$.
By the choice of $x_1x_2$ and $y_1y_2$, $z_1 \neq z_2$.
Suppose that  there exists a vertex $z_3$ other than $z_1,z_2$ in $V_2$.
If there is an arc from $x_1$ or $x_2$ to $z_3$, and there is an arc from $y_1$ or $y_2$ to $z_3$, then the edges $x_1x_2$ and $y_1y_2$ are adjacent to the same edge and we reach a contradiction.
Therefore either $(z_3,x_1), (z_3,x_2)$ are arcs in $D$ or $(z_3,y_1), (z_3,y_2)$ are arcs in $D$.
Without loss of generality, we may assume that  $(z_3,x_1), (z_3,x_2)$ are arcs in $D$.
Then $d^+_D(z_3) \ge 2$.
Since $d^+_D(x_1) \ge 1$, $d^+_D(x_2) \ge 1$, by Corollary~\ref{cor:edge}, $x_1z_3$ and $x_2z_3$ are edges in $C_{1,2}(D)$.
Since $y_1$ has an out-neighbor $z_2$ different from $z_3$, and $z_3$ has an out-neighbor $x_1$ different from $y_1$, $y_1z_3$ is an edge in $C_{1,2}(D)$ by Theorem~\ref{thm:notcompete}.
Then we have a path $x_1z_3y_1$ in $C_{1,2}(D)$ and reach a contradiction.
Therefore $V_2$ has exactly two vertices.
If one of $(x_1,z_2)$, $(x_2, z_2)$, $(y_1, z_1)$, $(y_2, z_1)$ is an arc in $D$, then, since $z_1$ and $z_2$ are common out-neighbors of $x_1$ and $x_2$, and $y_1$ and $y_2$, respectively, $x_1x_2$ and $y_1y_2$ are connected, which is a contradiction.
Therefore $(z_2,x_1)$, $(z_2,x_2)$, $(z_1,y_1)$, and $(z_1,y_2)$ are arcs in $D$.
Then, by Corollary~\ref{cor:edge}, $z_2x_1$, $z_2x_2$, $z_1y_1$, and $z_1y_2$ are edges in $C_{1,2}(D)$.
If $|V_1|=4$, then $C_{1,2}(D)$ consists of exactly two complete components since $x_1x_2$ and $y_1y_2$ belong to distinct components.
Suppose $|V_1| > 4$.
Take a vertex $u$ in $V_1$ other than $x_1,x_2,y_1,y_2$.
If both $(z_1, u)$ and $(z_2, u)$ are arcs in $D$, then $z_1z_2$ is an edge of $C_{1,2}(D)$ and so $x_1z_2z_1y_1$ is a path, which contradicts the fact that $x_1x_2$ and $y_1y_2$ belong to distinct components.
If both $(u,z_1)$ and $(u,z_2)$ are arcs in $D$, then $uz_1$ and $uz_2$ are edges in $C_{1,2}(D)$ by Corollary~\ref{cor:edge} and so $x_1z_2uz_1y_1$ is a path, which contradicts the fact that $x_1$ and $y_1$ belong to distinct components.
Therefore either $(u,z_1), (z_2,u)$ are arcs in $D$ or $(z_1,u), (u, z_2)$ are arcs in $D$.
Thus we have shown that every vertex in $V_1$ has indegree and outdegree both equal to $1$ in $D$.
This implies that a vertex in $V_1$ is an in-neighbor of $z_1$ if and only if it is an out-neighbor of $z_2$ in $D$.
Then $N^-_D(z_1) = N^+_D(z_2)$ and $N^-_D(z_2) = N^+_D(z_1)$.
Since $N^-_D(z_1)$ and $N^-_D(z_2)$ are cliques, $N^+_D(z_1)$ and $N^+_D(z_2)$ are cliques in $C_{1,2}(D)$.
Since $d^+_D(z_1) \ge 2$ and $d^+_D(z_2) \ge 2$, by Corollary~\ref{cor:edge}, $z_1$ is adjacent to every vertex in $N^+_D(z_1)$, and $z_2$ is adjacent to every vertex in $N^+_D(z_2)$ in $C_{1,2}(D)$.
Therefore $N^+_D(z_1) \cup \{z_1\}$ and $N^+_D(z_2) \cup \{z_2\}$ are cliques.
Note that $V(C_{1,2}(D))$ is partitioned into $N^+_D(z_1) \cup \{z_1\}$ and $N^+_D(z_2) \cup \{z_2\}$.
Since $x_1x_2$ and $y_1y_2$ belong to distinct components, $N^+_D(z_1) \cup \{z_1\}$ and $N^+_D(z_2) \cup \{z_2\}$ are the complete components in $C_{1,2}(D)$.

Now suppose that $x_1x_2$ or $y_1y_2$ links $V_1$ and $V_2$.
Without loss of generality, we may assume that $x_1x_2$ links $V_1$ and $V_2$, and $x_1 \in V_1$ and $x_2 \in V_2$.
Then, by Lemma~\ref{lem:neighbor}, $x_1$ has a neighbor in $V_1$ or $x_2$ has a neighbor in $V_2$.
Without loss of generality, we may assume that $x_1$ has a neighbor, say $w_1$, in $V_1$.
If $y_1y_2$ does not link $V_1$ and $V_2$, then we may apply the argument for the previous case to edges $x_1w_1$ and $y_1y_2$.
Suppose that $y_1y_2$ links $V_1$ and $V_2$.
Without loss of generality, we may assume that $y_1 \in V_1$ and $y_2 \in V_2$.
Now, by Lemma~\ref{lem:neighbor}, $y_1$ has a neighbor in $V_1$ or $y_2$ has a neighbor in $V_2$.
If $y_1$ has a neighbor, say $w_2$ in $V_1$, we may apply the argument for the previous case to edges $x_1w_1$ and $y_1w_2$.
If $y_2$ has a neighbor, say $w_3$ in $V_2$, we may apply the argument for the previous case to edges $x_1w_1$ and $y_2w_3$.
\end{proof}

\begin{Thm}\label{thm:diameter}
For a bipartite tournament $D$, the nontrivial component of $C_{1,2}(D)$ has diameter at most three.
\end{Thm}
\begin{proof}
Let $(V_1, V_2)$ be a bipartition of $D$.
Suppose, to the contrary, that $C_{1,2}(D)$ has a component of diameter at least four.
Then there are two vertices $u$ and $v$ whose distance is four.
Let $P$ be one of the shortest paths connecting $u$ and $v$.
By Proposition~\ref{prop:deg}, each vertex on $P$ has outdegree at least one in $D$.

First we consider the case where all the vertices of $P$ belong to $V_i$ for some $i \in \{1,2\}$.
Without loss of generality, we may assume $i=1$.
In this case, we let $P = uxyzv$.
By Proposition~\ref{prop:structure}, there is a vertex $w$ in $V_2$ such that $(y,w)$ and $(z,w)$ are arcs in $D$.
If either $(u,w)$ or $(v,w)$ is an arc in $D$, then either $uy$ or $yv$ is an edge in $C_{1,2}(D)$, which contradicts the fact that $P$ is a shortest path.
Therefore $(w,u)$ and $(w,v)$ are arcs in $D$ and so $d^+_D(w) \ge 2$.
Since $d^+_D(u) \ge 1$ and $d^+_D(v) \ge 1$, by Corollary~\ref{cor:edge}, $uw$ and $vw$ are edges in $C_{1,2}(D)$.
Then we have a path $uwv$ and reach a contradiction.

Before we take care of the remaining cases, we observe the following:
\begin{itemize}
  \item[(i)] Each vertex on $P$ has at most two of the remaining vertices as out-neighbors in $D$.
  \item[(ii)] Each vertex on $P$ has at most two of the remaining vertices as in-neighbors in $D$.
\end{itemize}
If there is a vertex $w_1$ on $P$ which has three of the remaining vertices as out-neighbors $w_2$, $w_3$, $w_4$ in $D$, then, by Corollary~\ref{cor:edge}, $w_1w_2$, $w_1w_3$, $w_1w_4$ form a $K_{1,3}$ in $C_{1,2}(D)$, which is impossible.
Thus (i) is true.
If there is a vertex $w_1$ on $P$ which has three of the remaining vertices as in-neighbors $w_2$, $w_3$, $w_4$ in $D$, then, $w_2$, $w_3$, $w_4$ form a triangle in $C_{1,2}(D)$, which contradicts the assumption that $P$ is a shortest path.
Thus (ii) is true.

Now we consider the case where one specific vertex on $P$ belongs to $V_i$ and the remaining four vertices on $P$ belong to $V_j$ for distinct $i$ and $j$.
Without loss of generality, we may assume that $i=1$.
We denote the specific vertex by $u_1$ and the remaining vertices by $v_1$, $v_2$, $v_3$, $v_4$.
Then $u_1$ belongs to $V_1$ and $v_1$, $v_2$, $v_3$, $v_4$ belong to $V_2$.
By (i) and (ii), exactly two of $v_1$, $v_2$, $v_3$, $v_4$ are out-neighbors of $u_1$ and the remaining two vertices are in-neighbors of $u_1$ in $D$.
We may assume that $v_1$, $v_2$ are out-neighbors of $u_1$ and $v_3$, $v_4$ are in-neighbors of $u_1$.
Then, by Corollary~\ref{cor:edge} again, $u_1v_1$ and $u_1v_2$ are edges in $C_{1,2}(D)$.
Furthermore, $u_1$ is a common out-neighbor of $v_3$ and $v_4$ in $D$, so $v_3v_4$ is an edge in $C_{1,2}(D)$.
Since $P$ is a shortest path of length four, one of $v_1, v_2$ should be adjacent to one of $v_3, v_4$.
Without loss of generality, we may assume that $v_2$ is adjacent to $v_3$.
Then $v_2$ and $v_3$ have a common out-neighbor in $V_1$ by Proposition~\ref{prop:structure}.
Since there is an arc from $u_1$ to $v_2$, a common out-neighbor of $v_2$ and $v_3$ cannot be $u_1$.
Thus $d^+_D(v_3) \ge 2$.
Hence, by Corollary~\ref{cor:edge}, $u_1v_3$ is an edge in $C_{1,2}(D)$.
Then $u_1, v_2, v_3$ form a triangle, which is impossible.

We consider the case where two vertices on $P$ belong to $V_i$ and the remaining three vertices on $P$ belong to $V_j$ for distinct $i$ and $j$.
Without loss of generality, we may assume $i=1$.
We denote by $u_1$, $u_2$ the vertices belonging to $V_1$ and by $v_1$, $v_2$, $v_3$ the vertices belonging to $V_2$.
By (i) and (ii), one or two of $v_1$, $v_2$, $v_3$ are out-neighbors of $u_i$ for each $i=1,2$.
Suppose that exactly two of $v_1$, $v_2$, $v_3$ are out-neighbors of $u_1$.
We may assume that $v_1, v_2$ are out-neighbors of $u_1$.
By Corollary~\ref{cor:edge}, $u_1v_1$ and $u_1v_2$ are edges in $C_{1,2}(D)$.
If $v_1$ or $v_2$ is an out-neighbor of $u_2$, $u_1$ is adjacent to $u_2$, however, $u_1$ is already adjacent to $v_1$ and $v_2$, which contradicts the assumption that $P$ is a shortest path.
Therefore both $v_1$ and $v_2$ are in-neighbors of $u_2$.
Then $v_1$ and $v_2$ are adjacent to form a triangle with $u_1v_1$ and $u_1v_2$, which is also impossible.
Therefore exactly one of $v_1$, $v_2$, $v_3$ is an out-neighbor of $u_1$.
By symmetry, exactly one of $v_1$, $v_2$, $v_3$ is an out-neighbor of $u_2$.
Suppose that $u_1$ and $u_2$ have one of $v_1$, $v_2$, $v_3$ as a common out-neighbor.
Then $u_1$ and $u_2$ are out-neighbors of  the remaining two vertices.
Since $u_1$ and $u_2$ have outdegree at least one, $u_1$, $u_2$ and one of the remaining two vertices form a triangle by Corollary~\ref{cor:edge} and we reach a contradiction.
Therefore $u_1$ and $u_2$ have distinct out-neighbors.
We may assume $v_1$ and $v_2$ are the out-neighbors of $u_1$ and $u_2$, respectively.
Then, by Corollary~\ref{cor:edge}, $u_1v_3$ and $u_2v_3$ are edges in $C_{1,2}(D)$.
Since $u_1$ and $u_2$ are common out-neighbors of $v_2,v_3$ and $v_1,v_3$, respectively, $v_2v_3$ and $v_1v_3$ are edges in $C_{1,2}(D)$.
Now $v_3$ is adjacent to $u_1$, $u_2$, $v_1$, and $v_2$, which is impossible.
Hence we have completed the proof.
\end{proof}

\section{$(1,2)$-step realizable graphs }

We say that a graph $G$ is \emph{$(1,2)$-step competition-realizable through a complete bipartite graph} if it is the $(1,2)$-step competition graph of a bipartite tournaments (in this paper, we only consider orientations of complete bipartite graphs and therefore we omit ``through complete bipartite graphs").

Theorem~\ref{thm:component} tells us that if a graph with more than one nontrivial component is $(1,2)$-step competition-realizable, then it must be a disjoint union of two nontrivial complete graphs.
Based upon this fact, it seems to be natural to ask which disjoint union of complete graphs are $(1,2)$-step competition-realizable.
The following proposition answers the question.

\begin{Prop}
Let $G$ be a disjoint union of $K_m$ and $K_n$ with $m \ge n$.
Then $G$ is $(1,2)$-step competition  realizable if and only if $n \neq 2$.
\end{Prop}
\begin{proof}
The `only if' part immediately follows from Theorem~\ref{thm:component}.

Now we show the `if' part.
If $n=1$, then the digraph with vertex set $\{u_1,u_2,\ldots, u_m, u \}$ and arc set $\{(u_i,u) \mid i=1,\ldots,m)\}$ is an orientation of ${K}_{m,1}$, and its $(1,2)$-competition graph is  isomorphic to $G$.
Suppose $n \ge 3$.
Let $D$ be the digraph with vertex set $\{v_1, v_2, \ldots, v_{m-1}, w_1, w_2, \ldots, w_{n-1} \} \cup \{v,w\}$ and arc set
\begin{align*}
&\{(v_i, w) \mid i=1,\ldots,m-1\} \cup
\{(w, w_i) \mid i=1,\ldots,n-1 \} \\ \cup \
&\{(w_i, v) \mid i=1,\ldots,n-1 \} \cup
\{(v, v_i) \mid i=1,\ldots,m-1 \}. \end{align*}
Then $D$ is an orientation of ${K}_{m+n-2, 2}$ with bipartition \[(\{v_1,\ldots,v_{m-1}, w_1, \ldots, w_{n-1}\}, \{v,w\})\] and it is easy to check that the  $(1,2)$-step competition graph of $D$ is isomorphic to $G$.
Therefore $G$ is $(1,2)$-step competition-realizable if $n \neq 2$.
\end{proof}

Let $G_1$ and $G_2$ be graphs with $m$ vertices and $n$ vertices, respectively.
The pair $(G_1,G_2)$ is said to be \emph{competition-realizable (through $K_{m,n}$}) if the disjoint union of $G_1$ and $G_2$ is
the competition graph of an orientation of the complete bipartite graph $K_{m,n}$
with bipartition $(V(G_1),V(G_2))$.

An \emph{edge clique cover} of a graph $G$ is a family $\mathcal{F}$ of cliques of $G$ such that, for any two adjacent vertices of $G$, there exists a clique in $\mathcal{F}$ that contains both of them.

Kim~{\it et al.}~\cite{Kim} gave the following result which characterizes certain competition-realizable pairs in terms of an edge clique cover.

\begin{Thm}[\hskip-0.0005em\cite{Kim}]\label{thm:charcomplete}
Let $G$ be a graph having no isolated vertices,
and let $m$ be a positive integer.
Then, $(G,K_m)$ is a competition-realizable pair
if and only if
there exists an edge clique cover $\mathcal{F}$ of $G$
of size at most $m$
such that
\[
|S \cup S'| \leq |V(G)|-1
\]
holds for any two cliques $S$ and $S'$ in $\mathcal{F}$.
\end{Thm}

Now we completely characterize competition-realizable pairs $(K_m, K_n)$, which extends the following results given by Kim~{\it et al.}~\cite{Kim}.
Then we eventually characterize the complete graphs which are $(1,2)$-step competition-realizable.

\begin{Prop}[\hskip-0.0005em\cite{Kim}]\label{prop:complete}
Let $m$ and $n$ be integers such that $m  \geq n \geq 6$. Then the pair $(K_m,K_n)$ is competition-realizable.
\end{Prop}

\begin{Prop}[\hskip-0.0005em\cite{Kim}]\label{prop:kim2}
Let $n$ be a positive integer.
Then the pair $(K_n, K_n)$ is competition-realizable if and only if $n = 1$ or $n \ge 6$.
\end{Prop}

\begin{Thm}\label{thm:complete3}
The following are true:
\begin{itemize}
  \item[(i)] For positive integers $m$ and $n$ with $m \ge n$, the pair $(K_m, K_n)$ is competition-realizable if and only if one of the following holds: $n=1$; $m \ge n \ge 6$; $m \ge 10$ and $n=5$.
  \item[(ii)] For positive integer $l$, the complete graph $K_{l}$ is $(1,2)$-step competition-realizable if and only if $l \ge 12$.
\end{itemize}
\end{Thm}
\begin{proof}
To show the `if' part of (i), it is sufficient to take care of the cases $n=1$; $m \ge 10$ and $n = 5$ by Proposition~\ref{prop:complete}.
If $n=1$, then we orient each edge in the complete bipartite graph $K_{m,1}$ so that each arc in the orientation goes from a vertex in the partite set with $m$ vertices to the vertex in the partite set with one vertex, and it is easy to check that the competition graph of the resulting orientation is $(K_{m}, K_1)$.
Now we will construct a digraph $D$ with vertex set $V(D) = X \cup Y$ where $|X| = m \ge 10$ and $|Y| = 5$ in the following way.
For each pair $\alpha$ of vertices in $Y$, we take a vertex $x_\alpha$ in $X$ so that $x_\alpha$ is a common out-neighbor of the vertices in $\alpha$ and $x_\alpha \neq x_\beta$ for distinct pairs $\alpha, \beta$.
Then we add additional arcs the direction of each of which is from $X$ to $Y$ so that the underlying graph of $D$ is a complete bipartite graph with bipartition $(X,Y)$.
Then each vertex in $X$ has outdegree at least three.
Since $|Y|=5$, each pair of vertices in $X$ has a common out-neighbor in $Y$ by the Pigeonhole Principle.
Therefore we have shown that $X$ and $Y$ form cliques in $C(D)$.
By Proposition~\ref{prop:structure}, there is no edge joining a vertex in $X$ and a vertex in $Y$.
Hence the competition graph of $D$ is the union of $K_m$ and $K_n$ and so $(K_m, K_n)$ is competition-realizable for $m \ge 10$ and $n=5$.

To show the `only if' part of (i), suppose that $(K_m, K_n)$ is competition-realizable for some $2 \le n \le 5$.
Then there exists a digraph $D$ whose competition graph is the union of $K_m$ and $K_n$.
Let $V_1 = V(K_m)$ and $V_2 = V(K_n)$.
By Theorem~\ref{thm:charcomplete}, there exists an edge clique cover $\mathcal{F}$ of $G := K_{n}$ with the vertex set $V_2$ such that $|\mathcal{F}| \le m$ and
\begin{equation}\label{eqn:clique}
|S \cup S'| \leq n - 1
\end{equation}
holds for any two cliques $S$ and $S'$ in $\mathcal{F}$.
Then each clique in $\mathcal{F}$ has size at most $n - 1$ by (\ref{eqn:clique}).
Suppose $\mathcal{F}$ has a clique $S_1$ of size $n - 1$.
Then there is a vertex $v$ in $V_2$ not in $S_1$.
Since $v$ is adjacent to every vertex in $S_1$ in $G$, there is a clique $S_2$ in $\mathcal{F}$ containing $v$ and at least one vertex in $S_1$.
Then $|S_1 \cup S_2| = n$ contradicting  (\ref{eqn:clique}).
Therefore every clique in $\mathcal{F}$ has size at most  $n - 2$.
Now we suppose that $\mathcal{F}$ has a clique $S_3$ of size $n - 2$.
Then there are two vertices in $V_2$ not in $S_3$.
Since they are adjacent in $G$, they must be contained in a clique $S_4$ in $\mathcal{F}$.
Then $|S_3 \cup S_4| = n$ contradicting  (\ref{eqn:clique}).
Therefore every clique in $\mathcal{F}$ has size at most  $n - 3$.
Since $n \ge 2$, there exists a clique in $\mathcal{F}$ of size at least two.
Therefore $n - 3 \ge 2$.
Thus, by our assumption that $V_2$ has at most five vertices, $n = 5$.
Hence each clique in $\mathcal{F}$ has size at most two.
However, there are ten edges in $K_n$ to be covered by the cliques in $\mathcal{F}$, so $m \ge 10$.

To show the `only if' part of (ii), suppose that the complete graph $K_{l}$ is $(1,2)$-step competition-realizable for some positive integer $l$.
Then, for some positive integers $m$ and $n$ satisfying $m+n=l$, there exists an orientation $D$ of $K_{m,n}$ such that $C_{1,2}(D) = K_l$.
By Proposition~\ref{prop:structure}, the competition graph of $D$ is the union of $K_m$ and $K_n$.
If $m=1$ or $n=1$, then it is easy to check that $C_{1,2}(D)$ has an isolated vertex by Proposition~\ref{prop:deg}.
Therefore $m, n \ge 2$ and so $m+n \ge 12$ by (i).

To show the `if' part of (ii), fix an integer $l \ge 12$.
Then there exist integers $m$ and $n$ such that $m+n=l$ and $m \ge n \ge 6$.
By Proposition~\ref{prop:structure} and Proposition~\ref{prop:complete}, there exists an orientation $D$ of $K_{m,n}$ with a bipartition $(V_1, V_2)$ whose $(1,2)$-step competition graph $G$ has the induced subgraphs $G[V_1]=K_m$ and $G[V_2]=K_n$. Thus we have to show that every vertex in $V_1$ is adjacent to every vertex in $V_2$ in $G$. By Corollary~\ref{cor:leasttwo}, it suffices to show that each vertex has outdegree at least two in $D$.

To the contrary, suppose that a vertex $u$ in $V_1$ has outdegree at most one in $D$.
If $u$ has no out-neighbor in $D$, then $u$ is an isolated vertex in $G[V_1]$, which contradicts $G[V_1]=K_m$.
Suppose that $u$ has the only out-neighbor $v$ in $D$.
Since $G[V_1]=K_m$, $v$ is also an out-neighbor of each vertex of $V_1$ in $D$.
Thus $v$ has no out-neighbor in $D$, and so $v$ is an isolated vertex in $G[V_2]$, which contradicts $G[V_2]=K_n$.
Thus every vertex in $V_1$ has outdegree at least two in $D$.
Similarly, we can show that every vertex in $V_2$ has outdegree at least two in $D$.
Therefore every vertex in $V_1$ is adjacent to every vertex in $V_2$ in $G$.
\end{proof}

\begin{Rem}
If a nontrivial component $K$ of the $(1,2)$-step competition graph $G$ of a bipartite tournament has diameter one, that is, it is complete, then by Theorems~\ref{thm:component} and~\ref{thm:complete3}, one of the following holds:
\begin{itemize}
    \item[(i)] $G$ is a complete graph with at least $12$ vertices;
    \item[(ii)] $K$ is a complete graph with at least two vertices and $G$ consists of $K$ and at least one isolated vertices;
    \item[(iii)] $G$ consists of two complete components of size at least three.
\end{itemize}
\end{Rem}

Now we consider the connected non-complete graphs which are $(1,2)$-step competition-realizable.
Then they have diameter $2$ or $3$ by Theorem~\ref{thm:diameter}.

\begin{Lem}\label{lem:share}
Let $D$ be a bipartite tournament containing the digraph given in Figure~\ref{fig:k14} as a subdigraph.
If the $(1,2)$-step competition graph of $D$ is connected and not isomorphic to $K_{1,4}$, then it has edge-sharing triangles.
\end{Lem}

\begin{proof}
The $(1,2)$-step competition graph of the digraph given in Figure~\ref{fig:k14} is a star graph $K_{1,4}$.
Therefore $G := C_{1,2}(D)$ contains $K_{1,4}$ as a subgraph.
Let $(V_1, V_2)$ be a bipartition of the complete bipartite graph from which $D$ is obtained.
For the vertices of the digraph given in Figure~\ref{fig:k14}, let $x$, $y$, $z$ be the vertices in $V_1$ and $w_1$, $w_2$ be the vertices in $V_2$.
Since $G$ is not isomorphic to $K_{1,4}$, there exists a vertex, say  $v$, distinct from $x$, $y$, $z$, $w_1$, $w_2$.

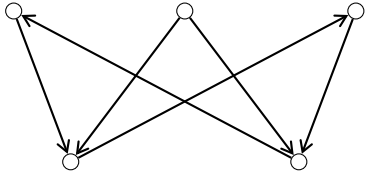
\begin{figure}
\begin{center}
\begin{tikzpicture}[x=1.5cm, y=2cm]
\begin{scope}

    \vertex (x1) at (0,1) [label=above:$$]{};
    \vertex (x2) at (1.5,1) [label=above:$$]{};
    \vertex (x3) at (3,1) [label=above:$$]{};

    \vertex (y1) at (0.5,0) [label=above:$$]{};
    \vertex (y2) at (2.5,0) [label=above:$$]{};

\end{scope}

    \draw[->,>=angle 60, thick] (x1) -- (y1);
    \draw[->,>=angle 60, thick] (x2) -- (y1);
    \draw[->,>=angle 60, thick] (x2) -- (y2);
    \draw[->,>=angle 60, thick] (x3) -- (y2);

    \draw[->,>=angle 60, thick] (y1) -- (x3);
    \draw[->,>=angle 60, thick] (y2) -- (x1);

\end{tikzpicture}

\end{center}
\caption{An orientation of ${K}_{3,2}$ }
\label{fig:k14}
\end{figure}

Assume $v \in V_1$.
Consider the case where $(w_1, v)$, $(w_2, v)$ are arcs in $D$.
Then $w_1w_2$ is an edge in $G$.
Since $G$ is connected, $d^+_D(v) \ge 1$ by Proposition~\ref{prop:deg}.
Moreover, each of $w_1$ and $w_2$ has an out-neighbor among $x$, $y$, $z$ other than $v$.
Therefore $vw_1$, $vw_2$ are edges in $G$ by Corollary~\ref{cor:edge}.
Thus $w_1w_2yw_1$ and $w_1w_2vw_1$ are triangles which share an edge and we are done.
Now we consider the case where $(v, w_1)$ or $(v, w_2)$ is an arc in $D$.
Without loss of generality, we may assume $(v, w_1)$ is an arc in $D$.
Then $\{v,x,y\}$ forms a triangle in $G$.
Since $v$ (resp.\ $w_2$) has an out-neighbor $w_1$ (resp.\ $x$) which is different from $w_2$ (resp.\ $v$), $vw_2$ is an edge in $G$ by Theorem~\ref{thm:notcompete}, which results in the triangle $yvw_2y$.
Then triangles $vxyv$ and $yvw_2y$ share an edge, which is desired.

Assume $v \in V_2$.
Consider the case where $(v, x)$ and $(v, z)$ are arcs in $D$.
Then $vx$ and $vz$ are edges in $G$ by Corollary~\ref{cor:edge}.
Since $v$ (resp.\ $y$) has an out-neighbor $x$ (resp.\ $w_1$) which is different from $y$ (resp.\ $v$),
$vy$ is an edge in $G$ by Theorem~\ref{thm:notcompete}.
Thus $vxyv$ and $vyzv$ are triangles in $G$ which share an edge as desired.
Now we consider the case where $(x, v)$ or $(z, v)$ is an arc in $D$.
Without loss of generality, we may assume $(x, v)$ is an arc in $D$.
Then $xv$ and $xw_1$ are edges in $G$ by Corollary~\ref{cor:edge}.
Suppose that $(v, z)$ is an arc in $D$.
Since $v$ (resp.\ $y$) has an out-neighbor $z$ (resp.\ $w_1$) which is different from $y$ (resp.\ $v$), $vy$ is an edge in $G$ by Theorem~\ref{thm:notcompete}.
Then $vxyv$ and $w_1xyw_1$ are triangles in $G$ which share an edge.
Suppose that $(z, v)$ is an arc in $D$.
Then $v$ is a common out-neighbor of $x$ and $z$ in $D$, so $xz$ is an edge in $G$.
In addition, $zw_2$ is an edge of $G$ by Corollary~\ref{cor:edge}.
Therefore $xyzx$ and $yzw_2y$ are triangles in $G$ which share an edge and this completes the proof.
\end{proof}

\begin{Thm}\label{thm:key}
Suppose that a connected non-complete graph $G$ does not have edge-sharing triangles.
Then $G$ is $(1,2)$-step competition-realizable if and only if it is isomorphic to the star graph $K_{1,4}$.
\end{Thm}
\begin{proof}
Let $D^*$ be a digraph given in Figure~\ref{fig:k14}.
Then $D^*$ is an orientation of $K_{3,2}$ and it is easy to check that the $(1,2)$-step competition graph of $D^*$ is isomorphic to $K_{1,4}$.
Therefore the `if' part is true.

To show the `only if' part, suppose that $G$ is the $(1,2)$-step competition graph of a bipartite tournament $D$ with a bipartition $(V_1, V_2)$.
Since $G$ is non-complete, $G$ has an induced path of length $2$.
We consider the following two cases:
\begin{itemize}
\item[(i)] there exists an induced path $P$ of length two in $G$ such that exactly one of the edges on $P$ joins $V_1$ and $V_2$.
\item[(ii)] there is no induced path of length two in $G$ such that exactly one of its edges joins $V_1$ and $V_2$.
\end{itemize}
We first consider the case (i).
Let $P = xyz$.
Without loss of generality, we may assume $x$ and $y$ are in $V_1$ and $z$ is in $V_2$.
Then $x$ and $y$ have a common out-neighbor $w$ in $V_2$.
We first consider the case $w = z$.
By Proposition~\ref{prop:deg}, $z$ has an out-neighbor, say $v$, in $V_1$ in $D$.
Since $y$ and $z$ belong to distinct partite sets and are adjacent, $y$ has an out-neighbor $u$ other than $w$ by Corollary~\ref{cor:edge}.
Since $x$ and $z$ belong to $V_i$ and $V_j$, respectively, for distinct $i$ and $j$, and are not adjacent, $u$ must be an in-neighbor of $x$ by the same corollary.
If $(u,v)$ is an arc in $D$, then $xyux$ and $yzuy$ are triangles which share an edge, which is impossible.
Therefore $(v,u)$ is an arc in $D$ and we obtain a subdigraph of $D$ isomorphic to the one given in Figure~\ref{fig:k14}.
Now we consider the case $w \neq z$.
Then, by Corollary~\ref{cor:edge}, $z$ cannot be an out-neighbor of $x$, for otherwise $x$ and $z$ are adjacent, which contradicts the hypothesis that $xyz$ is an induced path in $G$.
Therefore $(z,x)$ is an arc in $D$.
On the other hand, $w$ has an out-neighbor, say $v'$, by Proposition~\ref{prop:deg}.
Since $x$ and $z$ are not adjacent in $G$, $(y,z)$ and $(v',z)$ are arcs in $D$ by Corollary~\ref{cor:edge}, and we obtain a subdigraph of $D$ isomorphic to the one given in Figure~\ref{fig:k14}.
By the hypothesis, $G$ is a non-complete connected graph without edge-sharing triangles.
Therefore, by Lemma~\ref{lem:share}, $G$ is $K_{1,4}$ in both cases.

Now we consider the case (ii).
Since $G$ is connected, it has an edge linking $V_1$ and $V_2$.
Let $x_1y_1$ with $x_1 \in V_1$, $y_1 \in V_2$ be such an edge.
Then $x_1$ and $y_1$ $(1,2)$-compete by Corollary~\ref{cor:12compete}.
Without loss of generality, we may assume that $x_1$ and $y_1$ have a $(1,2)$-step common out-neighbor in $V_2$.
Then there exist $x_2 \in V_1$ and $y_2 \in V_2$ such that $(x_1,y_2)$, $(y_1,x_2)$, and $(x_2,y_2)$ are arcs in $D$, which yields the edge $x_1x_2$ in $G$.
Then $y_1x_1x_2$ is a path of length two in $G$.
By the case assumption, $y_1x_2$ is an edge in $G$.
Therefore $x_1x_2y_1x_1$ is a triangle in $G$.
If $x_1$ is adjacent to a vertex $z$ other than $x_2$, $y_1$ in $G$, then $z$ must be adjacent to at least one of $x_2$, $y_1$ in $G$ by the case assumption, which results in two triangles sharing an edge.
Therefore $x_2$ and $y_1$ are the only neighbors of $x_1$ in $G$.
For the same reason, $x_1$ and $y_1$ are the only neighbors of $x_2$ in $G$.

Suppose that there exists a vertex $y$ in $V_2$ such that $y$ and $y_1$ are adjacent in $G$.
Then $x_1y$ is an edge in $G$ by the case assumption.
Therefore $x_1y_1yx_1$ is a triangle which shares an edge with the triangle $x_1x_2y_1x_1$ in $G$, a contradiction.
Thus no vertex in $V_2$ is adjacent to $y_1$.
Hence $y_1$ and $y_2$ are not adjacent in $G$ and so $d_G(y_1, y_2) \ge 2$.
Take one of the shortest paths from $x_1$ to $y_2$ in $G$ and denote it by $Q$.
Since $x_2$ and $y_1$ are the only neighbors of $x_1$ in $G$, $x_2$ or $y_1$ immediately follows $x_1$ on $Q$.
However, $y_1$ and $x_1$ are the only neighbors of $x_2$, so $y_1$ must immediately follow $x_1$ since $Q$ is the shortest.
By the choice of $Q$, the $(y_1,y_2)$-section of $Q$ is a shortest path from $y_1$ to $y_2$.
Therefore $d_G(x_1, y_2) = d_G(y_1, y_2) + 1$.
By Theorem~\ref{thm:diameter}, $d_G(x_1, y_2) \le 3$ and so $d_G(y_1, y_2) \le 2$.
Thus $d_G(y_1,y_2) = 2$.
Hence there exists a vertex $x_3$ which is adjacent to $y_1$ and $y_2$.
As we have shown that no vertex in $V_2$ is adjacent to $y_1$, $x_3$ is in $V_1$.
Thus the edge $x_3y_1$ joins $V_1$ and $V_2$.
By Lemma~\ref{lem:neighbor}, $x_3$ has a neighbor in $G$, say $x_4$, in $V_1$.
Then $x_4y_1$ and $x_4y_2$ are edges in $G$ by the case assumption.
Therefore $x_3x_4y_1x_3$ and $x_3x_4y_2x_3$ are triangles which share an edge in $G$ and we reach a contradiction.
\end{proof}
\noindent
The following corollaries immediately follow from Theorem~\ref{thm:key}.
\begin{Cor}\label{cor:tree}
A tree is $(1,2)$-step competition-realizable if and only if it is isomorphic to the star graph $K_{1,4}$.
\end{Cor}

\begin{Cor}
If a connected graph is $(1,2)$-step competition-realizable, then it has a vertex of degree at least three.
\end{Cor}
\begin{proof}
If a connected $(1,2)$-step competition-realizable graph is triangle-free, then it is isomorphic to $K_{1,4}$ by Theorem~\ref{thm:key} and so the corollary is true.

Suppose that a connected $(1,2)$-step competition-realizable graph $G$ has a triangle.
If $G$ has less than four vertices, then it can easily be checked that $G$ has at most one edge and so it cannot have a triangle.
Therefore $G$ has at least four vertices.
Since $G$ is connected and has a triangle, it should have a vertex of degree at least three.
\end{proof}

\begin{Cor}
No connected unicyclic graph is $(1,2)$-step competition-realizable.
\end{Cor}

\begin{Cor}
If a connected graph with diameter three is $(1,2)$-step competition-realizable, then it has a triangle and edge-sharing cycles.
\end{Cor}

Remark: Figure~\ref{fig:diamter3}.

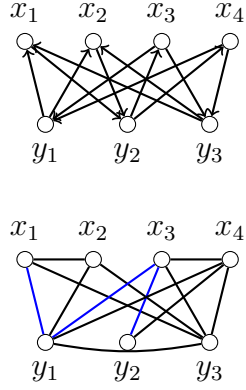
\begin{figure}
\begin{center}

\begin{tikzpicture}[x=0.9cm, y=1.1cm]
\begin{scope}

    \vertex (x1) at (0,1) [label=above:$x_1$]{};
    \vertex (x2) at (1,1) [label=above:$x_2$]{};
    \vertex (x3) at (2,1) [label=above:$x_3$]{};
    \vertex (x4) at (3,1) [label=above:$x_4$]{};

    \vertex (y1) at (0.3,0) [label=below:$y_1$]{};
    \vertex (y2) at (1.5,0) [label=below:$y_2$]{};
    \vertex (y3) at (2.7,0) [label=below:$y_3$]{};

\end{scope}
	\path 
    (x1) edge [<-,thick] (y1)
    (x1) edge [->,thick] (y2)
    (x1) edge [<-,thick] (y3)

    (x2) edge [<-,thick] (y1)
    (x2) edge [->,thick](y2)
    (x2) edge [<-,thick] (y3)

    (x3) edge [->,thick] (y1)
    (x3) edge [<-,thick] (y2)
    (x3) edge [->,thick] (y3)

    (x4) edge [->,thick] (y1)
    (x4) edge [<-,thick] (y2)
    (x4) edge [->,thick] (y3)
	;
\end{tikzpicture}

\end{center}

{\begin{center}
\begin{tikzpicture}[x=0.9cm, y=1.1cm]
\begin{scope}

    \vertex (x1) at (0,1) [label=above:${x_1}$]{};
    \vertex (x2) at (1,1) [label=above:$x_2$]{};
    \vertex (x3) at (2,1) [label=above:$x_3$]{};
    \vertex (x4) at (3,1) [label=above:$x_4$]{};

    \vertex (y1) at (0.3,0) [label=below:$y_1$]{};
    \vertex (y2) at (1.5,0) [label=below:${y_2}$]{};
    \vertex (y3) at (2.7,0) [label=below:$y_3$]{};

\end{scope}

	\path
    (x1) edge [thick] (x2)
    (x3) edge [thick] (x4)

    (y1) edge [thick, bend right=10] (y3)

    (x1) edge [thick,color=blue] (y1)
    (x1) edge [thick] (y3)

    (x2) edge [thick] (y1)
    (x2) edge [thick] (y3)

    (x3) edge [thick,color=blue] (y1)
    (x3) edge [thick,color=blue] (y2)
    (x3) edge [thick] (y3)

    (x4) edge [thick] (y1)
    (x4) edge [thick] (y2)
    (x4) edge [thick] (y3)

	;
\end{tikzpicture}
\end{center} }
\caption{A bipartite tournament and its $(1,2)$-step competition graph which has diameter three}
\label{fig:diamter3}
\end{figure}

\section{Extremal Cases}

In the following, we compute the maximum number of edges and the minimum number of edges which  the (1,2)-step competition graph of an orientation of $K_{m,n}$ might have.

Since a simple graph with $m+n$ vertices has at most ${m+n \choose 2}$ edges, any (1,2)-step competition graph $G$ of an orientation of $K_{m,n}$ satisfies
\[|E(G)|\leq {m+n \choose 2}.\]

It immediately follows from  Theorem~\ref{thm:complete3}(ii) that there is an orientation of $K_{m,n}$ for which its (1,2)-step competition graph has exactly ${m+n \choose 2}$ edges whenever $m \ge n \ge 6$.

\begin{Cor} For each integers $m$ and $n$ with $m \ge n \ge 6$, there exists an orientation $D$ of $K_{m,n}$ such that
\[
|E(C_{1,2}(D))| = {m+n \choose 2}.
\]
\end{Cor}

Now we consider the minimum case.

\begin{Thm}
Let $G$ be the $(1,2)$-step competition graph of an orientation of $K_{m,n}$ with edges as few as possible for positive integer $m,n$ with $m \ge n$. Then
\[
|E(G)|=
\begin{cases}
  0 & \mbox{if $m=n=2$;} \\
  {n \choose 2} & \mbox{otherwise}.
\end{cases}
\]
\end{Thm}
\begin{proof}
It is easy to see that the directed cycle on four vertices  is an orientation of $K_{2,2}$ and its $(1,2)$-step competition graph has no edge.
Thus the equality holds when $m=n=2$.

Let $D^*$ be a digraph with vertex set $\{u,v_1,v_2,\ldots,v_m\}$ and arc set $\{(u,v_i) \mid i = 1,2,\ldots,m\}$.
Then it can easily be check that $D^*$ is an orientation of $K_{1,m}$ and the (1,2)-step competition graph of $D^*$ has no edge.
Therefore $|E(G)| = 0 = {1 \choose 2}$. Thus the equality holds when $n=1$.

Suppose that $n \ge 2$ and $m \ge 3$.
Let $D^{**}$ be the orientation of $K_{m,n}$ defined by
\begin{align*}
V(D^{**}) &= \{u_1,u_2, \ldots, u_m, v_1, v_2, \ldots, v_n\} \\
A(D^{**}) &= \{(v_i,u_j) \mid i \in \{1,\ldots,n\}, j \in \{1,\ldots,m \} \}
\end{align*}
Then it is easy to check that the $(1,2)$-step competition graph of $D^{**}$ is the disjoint union of $K_n$ and the set of $m$ isolated vertices.
Therefore $|E(C_{1,2}(D^{**}))| = {n \choose 2}$ and so $|E(G)| \le {n \choose 2}$.
It remains to show $|E(G)| \ge {n \choose 2}$.
Let $D$ be an orientation of $K_{m,n}$ with bipartition $(V_1, V_2)$, $|V_1|=m$, $|V_2|=n$, whose (1,2)-step competition graph is $G$.
First, suppose that there exists a vertex $v \in V(D)$ with $d^+_D(v)=0$.
Then $d^-_D(v) = m$ or $d^-_D(v) = n$ and the in-neighborhood of $v$ forms a clique of size $m$ or $n$ in $G$.
Since $m \ge n$ by the hypothesis, $G$ has at least ${n \choose 2}$ edges and so $|E(G)| \ge {n \choose 2}$.
Now suppose that every vertex of $D$ has nonzero outdegree.
We show that there are at most $m+2$ vertices of outdegree $1$ by contradiction.
Suppose that $D$ has at least $m+3$ vertices of outdegree $1$.
Then there are two vertices $u_1, u_2$ of outdegree $1$ in $V_1$.
Then both $u_1$ and $u_2$ have indegree $n-1$.
Since $|V_2|=n$, there are at least $n-2$ common in-neighbors of $u_1$ and $u_2$, which implies that there are at least $n-2$ vertices of outdegree at least $2$.
Therefore there are at most $|V(D)| - (n-2) = m+2$ vertices of outdegree $1$ in $D$ and we reach a contradiction.
Thus we can conclude that $D$ has at most $m+2$ vertices of outdegree $1$.
Since each vertex of $D$ has outdegree at least $1$, it follows from Corollary~\ref{cor:edge} that, for any arc $(u,v)$ in $D$, $uv \notin E(G)$ if and only if $d^+_D(u) = 1$.
Therefore there are at least $mn - (m+2)$ edges in $G$.
By the way, if $n=2$, then $mn - (m+2) = m-2 \ge 1 = {2 \choose 2}$.
If $n \ge 3$, then
\begin{equation*}
mn - (m+2) = m(n-1) -2 \ge n(n-1) -2 \ge {n \choose 2}.
\end{equation*}
Thus $G$ has at least ${n \choose 2}$ edges.
As we have shown $|E(G)| \le {n \choose 2}$, it holds that $|E(G)| = {n \choose 2}$.
\end{proof}

\section*{Acknowledgments}

This research was supported by
the National Research Foundation of Korea (NRF)
grant funded by the Korea government (MEST) (No.~NRF-2015R1A2A2A01006885).
The first author's research was supported by Global Ph.D Fellowship Program
through the National Research Foundation of Korea (NRF)　
funded by the Ministry of Education (No.~NRF-2015H1A2A1033541).



\end{document}